\documentclass[a4paper,12pt]{article}
\usepackage{amsthm}
\usepackage{amsthm, amsmath, amssymb}
\usepackage{geometry}
\usepackage{mathrsfs}
\usepackage{caption}
\usepackage{color}
\usepackage{graphicx}
\newtheorem{theorem}{Theorem}

\newtheorem{corollary}{Corollary}

\newtheorem{example}{Example}

\newtheorem{problem}{Problem}
\newtheorem{remark}{Remark}

\newtheorem{definition}{Definition}
\usepackage[linesnumbered,ruled,vlined]{algorithm2e}
\usepackage{algorithmic}
\renewcommand\refname{Reference}
\usepackage{authblk}
\usepackage{hyperref}

\newsavebox\affbox
\author[a]{\textbf{Weifang Lv}}
\author[a]{\textbf{Quanyu Tang}}
\author[a]{\textbf{Wei Wang}\thanks{Corresponding author. Email address: wang\_weiw@xjtu.edu.cn}}
\author[b]{\textbf{Hao Zhang}} 

\affil[a]{ School of Mathematics and Statistics, Xi'an Jiaotong University, Xi'an, 710049, P.R. China}
\affil[b]{ School of Mathematics, Hunan University, Changsha 410082, P.R. China}
\title{\textbf{\LARGE Nearly permanental cospectral graphs}}
\date{}
\begin{document}
\renewcommand\refname{Reference}
\maketitle

\begin{abstract}
 Let $G$ be a simple graph of order $n$ with adjacency matrix $A= (a_{ij})$. The \emph{determinant} and the \emph{permanen}t of the matrix $A$ are defined as
\[\mathrm{det}A= \sum_{\sigma \in S_n}\mathrm{sgn}(\sigma) \prod_{i=1}^n a_{i\sigma(i)}\quad\text{and}\quad\mathrm{per}A= \sum_{\sigma \in S_n} \prod_{i=1}^n a_{i\sigma(i)},\]respectively. The polynomials $\phi(G;x) =\mathrm{det}(xI-A(G))$ and $\pi(G;x) =\mathrm{per}(xI-A(G))$ are called the \emph{characteristic polynomial} and the \emph{permanental polynomial} of $G$, respectively. Two graphs are said to be \emph{nearly cospectral} with respect to the determinant (resp. permanent) if the difference of their characteristic (resp. permanental) polynomials is a constant. Lv et al. introduced the nearly cospectral graphs problem with respect to the determinant, and provided partial results in the case modulo 4. In this paper, we mainly prove that the corresponding results also hold for the nearly cospectral graphs problem with respect to the permanent. The determinant and permanent are the immanants corresponding to the irreducible characters $(1^n)$ and $(n)$ of the symmetric group $ S_n $, respectively. Here, the \emph{immanant} $d_\lambda(A)$ of $A$ is defined as
\[d_\lambda(A) = \sum_{\sigma \in S_n} \chi_\lambda(\sigma) \prod_{i=1}^n a_{i\sigma(i)},\] where $\chi_\lambda$ is the irreducible character of $ S_n $ indexed by the partition $ \lambda $. The immanantal polynomial of $G$ associated with $ \chi_\lambda $ is given by $ \phi_\lambda(G;x)=d_\lambda(xI-A) $. In this paper, we also establish a similar result for nearly immanantal cospectral graphs in $\mathbb{F}_2[x]$ for all irreducible characters $\chi_\lambda$.\\

   \end{abstract}

\section{Introduction}
Let $S_n$ be the symmetric group on $n$ objects. The \emph{determinant} of an $n \times n$ real matrix $A = (a_{ij})$ is defined as
\[
\mathrm{det}A = \sum_{\sigma\in S_n}\mathrm{sgn}(\sigma) \prod_{i=1}^{n} a_{i\sigma(i)}.
\]
The \emph{permanent} of $A = (a_{ij})$ is defined as
\[
\mathrm{per}(A) = \sum_{\sigma\in S_n} \prod_{i=1}^{n} a_{i\sigma(i)}.
\] Let $A(G)$ be the adjacency matrix of a graph $G$ with $n$ vertices. The polynomial
\[\phi(G, x) = \det(xI - A(G)) = \sum_{k=0}^{n} a_k x^{n-k}\]is called the \emph{characteristic polynomial} of $G$ and the polynomial
\[\pi(G, x) = \mathrm{per}(xI - A(G)) = \sum_{k=0}^{n} b_k x^{n-k}\]is called the \emph{permanental polynomial} of $G$. The characteristic polynomial and the permanental polynomial are important among the well-studied graph polynomials.

 Two graphs are said to be \emph{cospectral with respect to the determinant (resp. permanent)} if they have the same characteristic (resp. permanental) polynomial. Two graphs $G$ and $H$ are \emph{cospectral modulo $p$ with respect to the determinant (resp. permanent)} if the corresponding coefficients of their characteristic (resp. permanental) polynomials differ by multiples of $p$. The \emph{complement} of a graph $G$, is the graph $\bar{G}$, with the same vertex set but whose edge set consists of the edges not present in $G$.

Ji, Tang, Wang and Zhang presented an invariant for cospectral graphs with respect to the determinant as follows.
\begin{theorem}[\cite{cospectral}]\label{cospectral}
    Let $G$ and $H$ be graphs with $\phi(G;x)=\phi(H;x)$, then $\phi(\bar{G};x)\equiv \phi(\bar{H};x)\pmod4$.
\end{theorem}
Spier presented an invariant for graphs that are cospectral modulo 4 with respect to the determinant, which strengthened Theorem \ref{cospectral} by Ji et al.
\begin{theorem}[\cite{Spier}]\label{mod4}
   For any graph $G$, given $\phi(G;x)\pmod4$, we can compute $\phi(\bar{G};x)\pmod4$.
\end{theorem}
It is surprising that the characteristic polynomial and the permanental polynomial of a graph are closely connected modulo 4.
\begin{theorem}\label{det per}
    For any graph $G$, $\phi(G;x)\pmod4$ and $\pi(G;x)\pmod{4}$ can be derived from each other.
\end{theorem}
By means of the relationship between $\phi(G;x)\pmod4$ and $\pi(G;x)\pmod{4}$, we can extend the relevant results on the characteristic polynomial to the permanental polynomial. Combining Theorems~\ref{mod4} and \ref{det per} yields a result analogous to Theorem~\ref{mod4} for the permanental polynomial.

\begin{theorem}\label{per4}
    For any graph $G$, given $\pi(G;x)\pmod{4}$, we can compute $\pi(\bar{G};x)\pmod{4}$.
\end{theorem}
 Theorem~\ref{per4} gives an invariant $\pi(\bar{G};x) \pmod{4}$ for graphs that are cospectral modulo 4 with respect to the permanent. Motivated by Theorem~\ref{mod4} and Theorem~\ref{per4} above, we introduce the (generalized) nearly cospectral graphs problem for the permanent, as a natural counterpart to the problem for the determinant~\cite{nearly}.
\begin{definition} \textup{Two graphs are said to be \emph{nearly cospectral with respect to the permanent} if the difference of their permanental polynomials is a constant, i.e., if $\pi(G;x)-\pi(H;x)=c$ holds for some constant $c$ (not necessarily zero). They are further termed \emph{generalized nearly cospectral with respect to the permanent} if they are nearly cospectral with respect to the permanent, and so are their complements, i.e., if $\pi(G;x)-\pi(H;x)=c$ and $\pi(\bar{G};x)-\pi(\bar{H};x)=d$ hold for some constants $c$ and $d$ (not necessarily zero).}
\end{definition}

So a natural question arises :

\begin{problem}\label{P}What can be said about the (generalized) nearly cospectral graphs with respect to the permanent?
\end{problem}
More precisely, we restrict ourselves to the following more modest questions:
 \begin{itemize}
 \item  If $\pi(G;x)$ and $\pi(H;x)$ differ by a constant, what's the relationship between $\pi(\bar{G};x)$ and $\pi(\bar{H};x)$?
     \item Moreover, if $\pi(G;x)$ and $\pi(H;x)$ differ by a constant, and $\pi(\bar{G};x)$ and $\pi(\bar{H};x)$ also differ by a constant, then what is the relationship between these two constants?
       \end{itemize}

For the determinant, Lv et al. gave the following results in the case of modulo 4~\cite{nearly}.
\begin{theorem}[\cite{nearly}]\label{A1}
  Let $G$ and $H$ be two graphs of odd order $n$. Suppose that $\phi(G;x)-\phi(H;x)\equiv c\pmod{4}$, where $c$ is a constant. Then $\phi(\bar{G};x)-\phi(\bar{H};x)\equiv c\pmod{4}$.
  \end{theorem}

\begin{theorem}[\cite{nearly}]\label{A2}
    Let $G$ and $H$ be two graphs of even order $n$. Suppose that $\phi(G;x)-\phi(H;x)\equiv c\pmod{4}$, where $c$ is a constant and $c\not\equiv0\pmod{4}$. Then $\phi(\bar{G};x)-\phi(\bar{H};x)\pmod{4}$ must be a polynomial of degree $\frac{n}{2}-1$.
\end{theorem}
\begin{theorem}[\cite{nearly}]\label{A3}
    Let $G$ and $H$ be two graphs of order $n\geq3$. Suppose that $\phi(G;x)-\phi(H;x)\equiv c\pmod{4}$ and $\phi(\bar{G};x)-\phi(\bar{H};x)\equiv d\pmod{4}$, where $c$ and $d$ are two constants, respectively. Then $c\equiv d\pmod{4}$. In particular, if $n$ is even, then $c\equiv d\equiv 0\pmod{4}$; and if $n$ is odd, then $c\equiv d\equiv 0~{\rm or}~2\pmod{4}$.
\end{theorem}

In contrast to the (generalized) nearly cospectral graphs problem with respect to the determinant, the following results are direct consequences of Theorem~\ref{det per} and Theorems~\ref{A1}, \ref{A2}, \ref{A3} for the permanent.
\begin{theorem}\label{B1}
  Let $G$ and $H$ be two graphs of odd order $n$. Suppose that $\pi(G;x)-\pi(H;x)\equiv c\pmod{4}$, where $c$ is a constant. Then $\pi(\bar{G};x)-\pi(\bar{H};x)\equiv c\pmod{4}$.
  \end{theorem}

\begin{theorem}\label{B2}
    Let $G$ and $H$ be two graphs of even order $n$. Suppose that $\pi(G;x)-\pi(H;x)\equiv c\pmod{4}$, where $c$ is a constant and $c\not\equiv0\pmod{4}$. Then $\pi(\bar{G};x)-\pi(\bar{H};x)\pmod{4}$ must be a polynomial of degree $\frac{n}{2}-1$.
\end{theorem}
\begin{theorem}\label{B3}
    Let $G$ and $H$ be two graphs of order $n\geq3$. Suppose that $\pi(G;x)-\i(H;x)\equiv c\pmod{4}$ and $\pi(\bar{G};x)-\pi(\bar{H};x)\equiv d\pmod{4}$, where $c$ and $d$ are two constants, respectively. Then $c\equiv d\pmod{4}$. In particular, if $n$ is even, then $c\equiv d\equiv 0\pmod{4}$; and if $n$ is odd, then $c\equiv d\equiv 0~{\rm or}~2\pmod{4}$.
\end{theorem}
Actually, the determinant and permanent are the immanants corresponding to the irreducible characters $(1^n)$ and $(n)$ of the symmetric group $ S_n $, respectively. Here, the immanant $d_\lambda(A)$ of $A$ is defined as
\[d_\lambda(A) = \sum_{\sigma \in S_n} \chi_\lambda(\sigma) \prod_{i=1}^n a_{i\sigma(i)},\] where $\chi_\lambda$ is the irreducible character of $ S_n $ indexed by the partition $ \lambda $. The immanantal polynomial of $G$ associated with $ \chi_\lambda $ is given by $ \phi_\lambda(G;x)=d_\lambda(xI-A) $. Two graphs are said to be \emph{nearly immanantal cospectral with respect to $ \chi_\lambda$} if the difference of their immanantal polynomials associated with $ \chi_\lambda $ is a constant, i.e., if $\phi_\lambda(G;x)-\phi_\lambda(H;x)=c$ holds for some constant $c$ (not necessarily zero).

Unlike the cases in Theorem~\ref{mod4}, Theorem~\ref{per4}, and Theorems~\ref{A1}--\ref{B3}, such properties do not hold for the immanantal polynomial $ \phi_\lambda(G;x)$ of $G$ with respect to general character $ \chi_\lambda $; see Section~\ref{ex}. However, similar results hold in $\mathbb{F}_2[x]$.
\begin{theorem}\label{mod2}
     For any graph $G$, given $\phi_\lambda(G;x)\pmod{2}$, we can compute $\phi_\lambda(\bar{G};x)\pmod{2}$.
\end{theorem}
\begin{theorem}\label{C1}
  Let $G$ and $H$ be two graphs of order $n$. Suppose that $\phi_\lambda(G;x)-\phi_\lambda(H;x)\equiv c\pmod{2}$, where $c$ is a constant. Then $\phi_\lambda(\bar{G};x)-\phi_\lambda(\bar{H};x)\equiv c\pmod{2}$.
  \end{theorem}
The paper is organized as follows. Section~\ref{S2} presents some preliminary results. Sections~\ref{S3} and \ref{S4} are devoted to the proofs of the main theorems of this paper. In the final section, we present some numerical examples.

\section{Preliminaries}\label{S2}
In this section, we present some necessary definitions, notations and preliminary results that will be used throughout the paper.
\subsection{Sachs' coefficient theorems}
 Let $ A(G) $ be the adjacency matrix of a graph $G$ with $ n $ vertices. The characteristic polynomial and the permanental polynomial of $ G $ are defined as
\[
\phi_{(1^n)}(G;x) = \det(xI - A(G)) = \sum_{k=0}^n a_k x^{n-k},
\]and\[
\phi_{(n)}(G;x) =\mathrm{per}(xI - A(G)) = \sum_{k=0}^n b_k x^{n-k},
\]respectively.

A \emph{Sachs graph} is a graph in which each component is a single edge or a cycle.  A Sachs subgraph of $G$ with $ k $ vertices is denoted by $ H_k $. A well-known theorem which relates the coefficients of $\phi_{(1^n)}(G;x)$ to the structural properties of graphs, which is called \emph{Sachs' Coefficient Theorem}~\cite{S1}. For the permanental polynomial $\phi_{(n)}(G;x)$ of a graph, there is a Sachs type result which was obtained by Merris et al.~\cite{S2}. Then
\begin{equation}\label{Sachs1}
    a_k = \sum_{H_k} (-1)^{p(H_k)} 2^{c(H_k)} \quad (1 \leq k \leq n),
\end{equation}
and
\begin{equation}\label{Sachs2}
b_k = (-1)^k \sum_{H_k} 2^{c(H_k)} \quad (1 \leq k \leq n),
\end{equation}where the summations range over all Sachs subgraphs $ H_k $ of $ G $, $ p(H_k) $ is the number of components of $ H_k $ and $ c(H_k) $ is the number of cycles of $ H_k $~\cite{S1,S2}.

Yu and Qu extended these two results to the immanantal polynomial~\cite{S3}.
\begin{theorem}[\cite{S3}]\label{Sachscharacter}
  Let $ G $ be a graph on $ n $ vertices and $ A(G) $ be its adjacency matrix. Let \[
\phi_{\lambda}(G;x) = \sum_{k=0}^{n}c_{\lambda,k}(G)x^{n-k}
\]be the immanantal polynomial of $ G $. Then $c_{\lambda,0}(G)=\chi_{\lambda}((1^n))$ and for $k\geq1$,
\begin{equation}\label{Sachs3}
    c_{\lambda,k}(G) = (-1)^k \sum_{B:|B|=k, B \subseteq V(G)} \sum_C \chi_{\lambda}(C) 2^{c(T_C(B))},
\end{equation}where $ B $ is a subset of $ V(G) $ with $ k $ elements, $ C $ is a conjugacy class of $ (S_n)_B $, and $(S_n)_B$ is the subgroup of $S_n$ defined as $(S_n)_B= \lbrace\sigma\in S_n\mid\sigma(i) = i\rbrace$ for any $i\in\lbrace 1,2,\dots,n\rbrace\backslash B$, $ T_C(B) $ is a graph whose components are single edges or cycles with vertex set $ B $ and which is determined by $ \sigma \in C $ and $ c(T_C(B)) $ is the number of cycles in $ T_C(B) $.
\end{theorem}
\begin{corollary}\label{odd k}
    For any graph $G$, $c_{\lambda,k}(G)$ is even for every odd $k$.
\end{corollary}
\begin{proof}
    Suppose that $k$ is odd. By Theorem~\ref{Sachscharacter}, only Sachs subgraphs of $G$ with $k$ vertices contribute to $c_{\lambda,k}(G)$. If $k=1$, then $c_{\lambda,k}(G)=0$. If $k\geq3$, then every such Sachs subgraph contains at least one cycle. Otherwise, all components in such a Sachs subgraph are edges. This implies that the Sachs subgraph has an even number of vertices, which contradicts the fact that $k$ is odd. Therefore, it follows from Eq.~\eqref{Sachs3} that $c_{\lambda,k}(G)$ is even.
\end{proof}
\subsection{Permanent and perfect matchings}
Let $G$ be a graph with $n$ vertices. An \emph{$r$-matching} in $G$ is a set of $r$ edges, no two of which have a vertex in common. A \emph{perfect matching} in $G$ is a $\frac{n}{2}$-matching. There is a close relationship between the permanent of the adjacency matrix and the number of perfect matchings in the graph.
\begin{theorem}[\cite{perK}]\label{K^2+odd}
Let $K(G)$ denote the number of perfect matchings of the graph $G$. Then we have
 \[ \mathrm{per}A(G) = K(G)^2 + \sum_{H_n^*} 2^{c(H_n^*)},\]where $H_n^*$ denotes the Sachs spanning subgraph containing at least one odd cycle, and $c(H_n^*)$ denotes the number of cycles in $H_n^*$. In particular, for a bipartite graph $G$, since $G$ contains no odd cycles, we have \[ \mathrm{per} A(G) = K(G)^2. \]
\end{theorem}
\begin{corollary}\label{perA}
    Let $G$ be a graph of order $n$, and let $\pi(G;x)=\sum_{k=0}^n b_k x^{n-k}$ be the permanental polynomial of $G$. Then the following hold:
    \begin{enumerate}
        \item $b_k$ is even for every odd $k$. In particular, if $G$ has odd order, then $\mathrm{per}A(G)$ is even;
    \item If $n$ is even, then $\mathrm{per}A(G)\equiv0$ or $1\pmod{4}$.\end{enumerate}
\end{corollary}
\begin{proof}
    As a special case of the immanantal polynomial, the first conclusion clearly holds for the permanental polynomial by Corollary~\ref{odd k}. In particular, if $G$ is a graph of odd order, then $b_n$ is even. That is, $(-1)^n\mathrm{per}A(G)$ is even.

    Suppose that $G$ is a graph of even order. If $G$ contains a Sachs spanning subgraph $H_n^*$ with at least one odd cycle, then $H_n^*$ must contain an even number of odd cycles, since the components of $H_n^*$, apart from odd cycles, are even cycles and edges. Therefore, we have $\mathrm{per}A(G)\equiv  K(G)^2\pmod4$ by Theorem~\ref{K^2+odd}. More precisely, $\mathrm{per}A(G)\equiv0\pmod4$ if $K(G)$ is even, and $\mathrm{per}A(G)\equiv1\pmod4$ if $K(G)$ is odd.
\end{proof}
\begin{remark}
    \textup{It follows from Corollary~\ref{perA} that $\mathrm{per}A(G)\equiv0$ or $1\pmod{4}$ when $n$ is even. For graphs $G$ and $H$ with even order, if $\pi(G;x)-\pi(H;x)\equiv c\pmod{4}$, then we have $c\equiv0$ or $\pm1\pmod4$.}
\end{remark}
\begin{remark}
    \textup{For a graph $G$ with an even number of vertices, the permanent and the determinant of its adjacency matrix $A$ exhibit different properties. For the determinant, we have $\det(A)\equiv 0$ or $1\pmod{4}$ if $n \equiv 0\pmod{4}$ and $\det(A)\equiv 0$ or $-1 \pmod{4}$ if $n \equiv 2 \pmod{4}$~\cite{detA}, which contrasts with the corresponding result for the permanent in Corollary~\ref{perA}.}
\end{remark}
\subsection{The matching polynomial}
Recall that an $r$-matching in a graph $G$ is a set of $r$ edges, no two of which have a vertex in common. The number of $r$-matchings in $G$ is denoted by $m(G,r)$. We set $m(G,0) = 1 $ and define the matchings polynomial of $ G $ by (see~\cite{matchc}, chapter 1)\begin{equation}\label{mu}
    \mu(G,x) := \sum_{r=0}^{\lfloor n/2\rfloor} (-1)^r m(G,r) x^{n-2r}.
\end{equation}
In particular, the number of $r$-matchings in the complete graph $K_n$ on $n$ vertices~\cite{matchc} is \[m(K_n, r) = \frac{n!}{r!(n - 2r)! \, 2^r}.\]
Godsil presented the relationship between the matching polynomial of $G$ and the matching polynomial of $\bar{G}$~\cite{matching}.
\begin{theorem}[\cite{matching}]\label{matchP}
Let $G$ be a graph with $n$ vertices. Then
    \[\mu(\bar{G}, x) = \sum_{r=0}^{\lfloor n/2 \rfloor} m(G,r)\mu(K_{n-2r}, x).\]
\end{theorem}
Furthermore, we establish a relationship between the matching numbers of $G$ and $\bar{G}$.
\begin{corollary}\label{matchn}
    Let $G$ be a graph with $n$ vertices. Then
    \[ (-1)^r m(\bar{G},r)=\sum_{i=0}^{r} (-1)^{r-i}m(G,i)m(K_{n-2i},r-i). \]
\end{corollary}
\begin{proof}
It follows from Eq.~\eqref{mu} and Theorem~\ref{matchP} that \[\sum_{r=0}^{\lfloor n/2\rfloor} (-1)^r m(\bar{G},r) x^{n-2r}=\sum_{r=0}^{\lfloor n/2 \rfloor} m(G,r)\sum_{i=0}^{\lfloor n/2\rfloor-r} (-1)^i m(K_{n-2r},i) x^{n-2r-2i}.\]By comparing the coefficients of $x^{n-2r}$ on both sides of the equation, we completes the proof.
\end{proof}
\section{Nearly cospectral graphs with respect to the permanent}\label{S3}
In this section, we present proofs of Theorems~\ref{det per}, \ref{per4}, \ref{B1}, \ref{B2}, and \ref{B3}. Theorems~\ref{per4}, \ref{B1}, \ref{B2}, and \ref{B3} provide partial answers to the (generalized) nearly cospectral graphs problem modulo 4 with respect to the permanent, and Theorem~\ref{det per} is the key result used in their proofs.

Firstly, we show that $\phi(G;x)\pmod4$ and $\pi(G;x)\pmod{4}$ can be derived from each other.
\begin{proof}[Proof of Theorem~\ref{det per}]
  For each Sachs subgraph $H_k$ with $k$ vertices of $G$, the contributions of $H_k$ to $a_k$ and $b_k$ are $(-1)^{p(H_k)} 2^{c(H_k)}$ and $(-1)^k 2^{c(H_k)}$, respectively.

If $c(H_k)\geq2$, then \[(-1)^{p(H_k)} 2^{c(H_k)}\equiv(-1)^{\lfloor k/2\rfloor}(-1)^k 2^{c(H_k)}\equiv0\pmod4.\]
If $c(H_k)=1$, then \[(-1)^{p(H_k)} 2^{c(H_k)}\equiv(-1)^{\lfloor k/2\rfloor}(-1)^k 2^{c(H_k)}\equiv2\pmod4.\]
If $c(H_k)=0$, then all components in $H_k$ are edges. Hence, we get $k=2p(H_k)$ is even. Therefore, \[(-1)^{p(H_k)} 2^{c(H_k)}\equiv(-1)^{\lfloor k/2\rfloor}(-1)^k 2^{c(H_k)}\pmod4.\]
Therefore, for each $H_k$, we have \[(-1)^{p(H_k)} 2^{c(H_k)}\equiv(-1)^{\lfloor k/2\rfloor}(-1)^k 2^{c(H_k)}\pmod4.\]
It follows from Eq.~\eqref{Sachs1} and Eq.~\eqref{Sachs2} that we obtain \begin{equation}\label{(-1)^k/2}
    a_k\equiv(-1)^{\lfloor k/2\rfloor}b_k\pmod4.
\end{equation}
Therefore, $\phi(G;x)\pmod4$ and $\pi(G;x)\pmod{4}$ can be derived from each other.
\end{proof}
Naturally, by means of Theorem~\ref{det per}, the results of Theorems~\ref{mod4}, \ref{A1}, \ref{A2}, and \ref{A3} can be extended to Theorems~\ref{per4}, \ref{B1}, \ref{B2}, and \ref{B3}.

\begin{proof}[Proof of Theorem~\ref{per4}]
    Let $G$ and $H$ be graphs that satisfy \[\pi(G;x)\equiv\pi(H;x)\pmod4.\]
    It follows from Eq.~\eqref{(-1)^k/2} that \[\phi(G;x)\equiv\phi(H;x)\pmod4.\]
    According to Theorem~\ref{mod4}, we have \[\phi(\bar{G};x)\equiv\phi(\bar{H};x)\pmod4.\]
    Applying Eq.~\eqref{(-1)^k/2} again, we obtain  \[\pi(\bar{G};x)\equiv\pi(\bar{H};x)\pmod4.\]
    This completes the proof.
\end{proof}
\begin{proof}[Proof of Theorem~\ref{B1}]
Let $G$ and $H$ be two graphs of odd order. Suppose that \[\pi(G;x)-\pi(H;x)\equiv c\pmod{4},\] where $c$ is a constant. It follows from Eq.~\eqref{(-1)^k/2} that \[\phi(G;x)-\phi(H;x)\equiv (-1)^{\lfloor n/2\rfloor} c\pmod{4}.\]
According to Theorem~\ref{A1}, we have \[\phi(\bar{G};x)-\phi(\bar{H};x)\equiv (-1)^{\lfloor n/2\rfloor} c\pmod{4}.\] Applying Eq.~\eqref{(-1)^k/2} once more, we get\[\pi(\bar{G};x)-\pi(\bar{H};x)\equiv c\pmod{4}.\]
This completes the proof.
\end{proof}
\begin{proof}[Proof of Theorem~\ref{B2}]
Let $G$ and $H$ be two graphs of even order. Suppose that \[\pi(G;x)-\pi(H;x)\equiv c\pmod{4},\] where $c$ is a constant and $c\not\equiv0\pmod4$. It follows from Eq.~\eqref{(-1)^k/2} that \[\phi(G;x)-\phi(H;x)\equiv (-1)^{\lfloor n/2\rfloor} c\pmod{4}.\]
According to Theorem~\ref{A2}, we have $\phi(\bar{G};x)-\phi(\bar{H};x)\pmod{4}$ is a polynomial of degree $\frac{n}{2}-1$. Applying Eq.~\eqref{(-1)^k/2} once more, we get $\pi(\bar{G};x)-\pi(\bar{H};x)\pmod{4}$ is a polynomial of degree $\frac{n}{2}-1$.
This completes the proof.
\end{proof}
\begin{proof}[Proof of Theorem~\ref{B3}]
 From Theorem~\ref{B1}, the conclusion clearly holds when $n$ is odd.

Suppose that $n$ is even. If $c\equiv0\pmod{4}$, it follows from Theorem~\ref{per4} that $d\equiv0\pmod{4}$. According to Theorem~\ref{B2}, if $c\not\equiv0\pmod{4}$, then $\pi(\bar{G};x)-\pi(\bar{H};x)\pmod{4}$ must be a polynomial of degree $\frac{n}{2}-1$. Moreover, when $n\geq4$, $\pi(\bar{G};x)-\pi(\bar{H};x)\pmod{4}$ is a non-constant polynomial. Therefore, when $n\geq4$, if $\pi(G;x)-\pi(H;x)\equiv c\pmod{4}$ and $\pi(\bar{G};x)-\pi(\bar{H};x)\equiv d\pmod{4}$, then we must have $c\equiv d\equiv0\pmod{4}$. This completes the proof.
\end{proof}
\section{Nearly immanantal cospectral graphs with respect to $\chi_\lambda$}\label{S4}
In this section, we present proofs of Theorem~\ref{mod2} and Theorem~\ref{C1}, which provide partial answers to the (generalized) nearly immanantal cospectral graphs problem with respect to $\chi_\lambda$ in $\mathbb{F}_2[x]$.
\begin{proof}[Proof of Theorem~\ref{mod2}]
    Let $ G $ be a graph on $ n $ vertices and $ A(G) $ be its adjacency matrix. Write the immanantal polynomial of $ G $ as \[
\phi_{\lambda}(G;x) = \sum_{k=0}^{n}c_{\lambda,k}(G)x^{n-k}.
\] It follows from Corollary~\ref{odd k} that \[\phi_{\lambda}(G;x) \equiv \sum_{j=0}^{\lfloor n/2 \rfloor} c_{\lambda,2j}(G) x^{n-2j}\pmod2.\]

It follows from Theorem~\ref{Sachscharacter} that for $j\geq1$,\[c_{\lambda,2j}(G)=\sum_{B:|B|=2j, B \subseteq V(G)} \sum_C \chi_{\lambda}(C) 2^{c(T_C(B))}\equiv \chi_{\lambda}((2^j))m(G,j)\pmod2,\] where $m(G,j)$ is the number of $j$-matchings in $G$. In fact, only those Sachs subgraphs of $G$ on $k$ vertices whose components are all edges affect $c_{\lambda,2j}(G)\pmod2$. That is, only the conjugacy classes of type $(2^j)$ in $ (S_n)_B $ contribute to $c_{\lambda,2j}(G)\pmod2$. Note that $c_{\lambda,0}(G)=\chi_{\lambda}((1^n))$ and $m(G,0)=1$, we identify the two conjugacy classes $(2^0)$ and $(1^n)$ in $S_n$ as equivalent. Therefore, for $j\geq0$, we have \[c_{\lambda,2j}(G)\equiv \chi_{\lambda}((2^j))m(G,j)\pmod2.\]

By Corollary~\ref{matchn}, we have \begin{align}\label{G barG mod2}
     c_{\lambda,2j}(\bar{G})&\equiv \chi_{\lambda}((2^j))m(\bar{G},j)\nonumber\\&\equiv\chi_{\lambda}((2^j))\left\lbrack\sum_{i=0}^{j}m(G,i)m(K_{n-2i},j-i)\right\rbrack\nonumber\\&\equiv\chi_{\lambda}((2^j))\left\lbrack\sum_{i=0}^{j} \frac{c_{\lambda,2i}(G)}{\chi_{\lambda}((2^i))}m(K_{n-2i},j-i)\right\rbrack.\pmod2
\end{align}
Therefore, given $\phi_{\lambda}(G;x)\pmod2$, we can obtain $\phi_{\lambda}(\bar{G};x)\pmod2$.
\end{proof}
Suppose that $\phi_\lambda(G;x)-\phi_\lambda(H;x)\equiv c\pmod{2}$, where $c$ is a constant. In particular, the case $c\equiv0\pmod2$ corresponds exactly to Theorem~\ref{mod2}. We now prove the general case of $c$.
\begin{proof}[Proof of Theorem~\ref{C1}]
    Suppose that $\phi_\lambda(G;x)-\phi_\lambda(H;x)\equiv c\pmod{2}$, where $c$ is a constant. If $n$ is odd, it follows from Corollary~\ref{odd k} that $c\equiv0\pmod2$. By Theorem~\ref{mod2}, we have $\phi_\lambda(\bar{G};x)-\phi_\lambda(\bar{H};x)\equiv c\pmod{2}$.

    If $n$ is even, note that \[c_{\lambda,2j}(G) \equiv c_{\lambda,2j}(H) \pmod{2}\quad\text{for}\quad j \in \{0, 1, 2, \dots, \frac{n}{2} - 1\}\]and \[c_{\lambda,n}(G) - c_{\lambda,n}(H) \equiv c \pmod{2}.\] It follows from Eq.~\eqref{G barG mod2} that \[c_{\lambda,2j}(\bar{G})\equiv c_{\lambda,2j}(\bar{H})\quad\text{for}\quad j\in\lbrace0,1,2,\dots,\frac{n}{2}-1\rbrace\] and \[c_{\lambda,n}(\bar{G})-c_{\lambda,n}(\bar{H})\equiv c_{\lambda,n}(G)-c_{\lambda,n}(H)\equiv c\pmod2.\]This completes the proof.
\end{proof}
\section{Examples}\label{ex}
In this section, we demonstrate that several results previously established for the determinant (Theorems~\ref{mod4}, \ref{A1}, \ref{A2}, and \ref{A3}) and the permanent (Theorems~\ref{per4}, \ref{B1}, \ref{B2}, and \ref{B3}) no longer hold for a general character, by examining the second immanant.

The \emph{second immanant} of an $n\times n$ matrix $A$ is the immanant corresponding to the irreducible character $(2,1^{n-2})$ of $S_n$. Let $G$ be a graph of order $n$ with adjacency matrix $A$. The \emph{second immanantal polynomial} $ \phi_{(2,1^{n-2})}(G;x)$ of $G$ associated with the character $(2,1^{n-2})$ is denoted by
\[
\phi_{(2,1^{n-2})}(G;x) = d_{(2,1^{n-2})}(xI - A).
\]Let $X=(x_{ij})$ be an $n\times n$ matrix $(k\geq2)$, and let $X_i$ be the submatrix of $X$ obtained by deleting the $i$-th row and the $i$-th column. Merris~\cite{second} gave that\begin{equation}\label{2}
  d_{(2,1^{n-2})}(X) = \sum_{i=1}^n x_{ii} \det X_i-\det X.
\end{equation}
Therefore, we can compute $\phi_{(2,1^{n-2})}(G;x)$ by the characteristic polynomial.

We now give counterexamples for graphs with an odd number of vertices and for graphs with an even number of vertices, respectively.
\begin{example}\label{ex1}{\textup {Let $n=9$. $G_1,G_2,G_3$, and $G_4$ are graphs for which the first $n$ coefficients of the second immanantal polynomials are congruent modulo 4; see Figure~\ref{9}. Their second immanantal polynomials are \[\phi_{(2,1^{n-2})}(G_1;x)=8x^9 - 114x^7 - 120x^6 + 256x^5 + 348x^4 - 26x^3 - 82x^2 - 8\equiv2 x^7 + 2 x^3 + 2 x^2\,(\mathrm{mod}\,4),\]
\[\phi_{(2,1^{n-2})}(G_2;x)=8x^9 - 102x^7 - 80x^6 + 228x^5 + 168x^4 - 126x^3 - 38x^2 + 2\equiv2 x^7 + 2 x^3 + 2 x^2+2\,(\mathrm{mod}\,4),\]
\[\phi_{(2,1^{n-2})}(G_3;x)=8x^9 - 114x^7 - 100x^6 + 216x^5 + 156x^4 - 78x^3 - 14x^2\equiv2 x^7 + 2 x^3 + 2 x^2\,(\mathrm{mod}\,4),\]
\[\phi_{(2,1^{n-2})}(G_4;x)=8x^9 - 102x^7 - 80x^6 + 248x^5 + 228x^4 - 98x^3 - 62x^2 - 12\equiv2 x^7 + 2 x^3 + 2 x^2\,(\mathrm{mod}\,4).\]
The second immanantal polynomials of their complements are \[\phi_{(2,1^{n-2})}(\bar{G_1};x)=8x^9 - 102x^7 - 60x^6 + 196x^5 + 72x^4 - 100x^3 - 8x^2 + 2\equiv2 x^7 + 2\,(\mathrm{mod}\,4),\]
\[\phi_{(2,1^{n-2})}(\bar{G_2};x)=8x^9 - 114x^7 - 100x^6 + 248x^5 + 180x^4 - 160x^3 - 38x^2 + 8\equiv2 x^7 + 2 x^2\,(\mathrm{mod}\,4),\]
\[\phi_{(2,1^{n-2})}(\bar{G_3};x)=8x^9 - 102x^7 - 80x^6 + 236x^5 + 210x^4 - 112x^3 - 74x^2 - 18\equiv2 x^7+ 2 x^4+ 2 x^2 +2\,(\mathrm{mod}\,4),\]
\[\phi_{(2,1^{n-2})}(\bar{G_4};x)=8x^9 - 114x^7 - 90x^6 + 228x^5 + 132x^4 - 120x^3 - 18x^2\equiv2 x^7 + 2 x^6 + 2 x^2\,(\mathrm{mod}\,4).\]}}
\end{example}

\begin{example}\label{ex2}{\textup {Let $n=8$. $G_1,G_2$ are graphs for which the first $n$ coefficients of the second immanantal polynomials are congruent modulo 4; see Figure~\ref{81}. Their second immanantal polynomials are \[\phi_{(2,1^{n-2})}(G_1;x)=7x^8 - 65x^6 - 56x^5 + 48x^4 + 24x^3 - 8x^2\equiv3 x^8 + 3 x^6\,(\mathrm{mod}\,4),\]
\[\phi_{(2,1^{n-2})}(G_2;x)=7x^8 - 65x^6 - 56x^5 + 60x^4 + 44x^3 - 8x^2 - 1\equiv3 x^8 + 3 x^6+3\,(\mathrm{mod}\,4).\]
The second immanantal polynomials of their complements are \[\phi_{(2,1^{n-2})}(\bar{G_1};x)=7x^8 - 75x^6 - 104x^5 + 9x^4 + 56x^3 + 11x^2\equiv3 x^8 + x^6 + x^4 + 3 x^2\,(\mathrm{mod}\,4),\]
\[\phi_{(2,1^{n-2})}(\bar{G_2};x)=7x^8 - 75x^6 - 88x^5 + 45x^4 + 64x^3 - x^2 - 1\equiv3 x^8 + x^6 + x^4 + 3 x^2+3\,(\mathrm{mod}\,4).\]
Therefore, we have \[\phi_{(2,1^{n-2})}(G_1;x)-\phi_{(2,1^{n-2})}(G_2;x)\equiv1\pmod4\] and \[\phi_{(2,1^{n-2})}(\bar{G_1};x)-\phi_{(2,1^{n-2})}(\bar{G_2};x)\equiv1\pmod4.\]}}

{\textup{Let $n=8$. $G_3,G_4$ are graphs for which the first $n$ coefficients of the second immanantal polynomials are congruent modulo 4; see Figure~\ref{82}. Their second immanantal polynomials are \[\phi_{(2,1^{n-2})}(G_3;x)=7x^8 - 80x^6 - 64x^5 + 84x^4 + 32x^3 - 4x^2\equiv3 x^8\,(\mathrm{mod}\,4),\]
\[\phi_{(2,1^{n-2})}(G_4;x)=7x^8 - 60x^6 - 24x^5 + 96x^4 + 20x^3 - 28x^2 - 1\equiv3 x^8 +3\,(\mathrm{mod}\,4).\]
The second immanantal polynomials of their complements are \[\phi_{(2,1^{n-2})}(\bar{G_3};x)=7x^8 - 60x^6 - 48x^5 + 66x^4 + 64x^3 + 3\equiv3 x^8+2 x^4 +3 \,(\mathrm{mod}\,4),\]
\[\phi_{(2,1^{n-2})}(\bar{G_4};x)=7x^8 - 80x^6 - 72x^5 + 102x^4 + 72x^3 - 26x^2 - 8\equiv3 x^8 + 2 x^4+2 x^2\,(\mathrm{mod}\,4).\]
Therefore, we have \[\phi_{(2,1^{n-2})}(G_3;x)-\phi_{(2,1^{n-2})}(G_4;x)\equiv1\pmod4\] and \[\phi_{(2,1^{n-2})}(\bar{G_3};x)-\phi_{(2,1^{n-2})}(\bar{G_4};x)\equiv2x^2+1\pmod4.\]
}}
\end{example}
\begin{figure}[htbp]
\centering
\includegraphics[width=10cm]{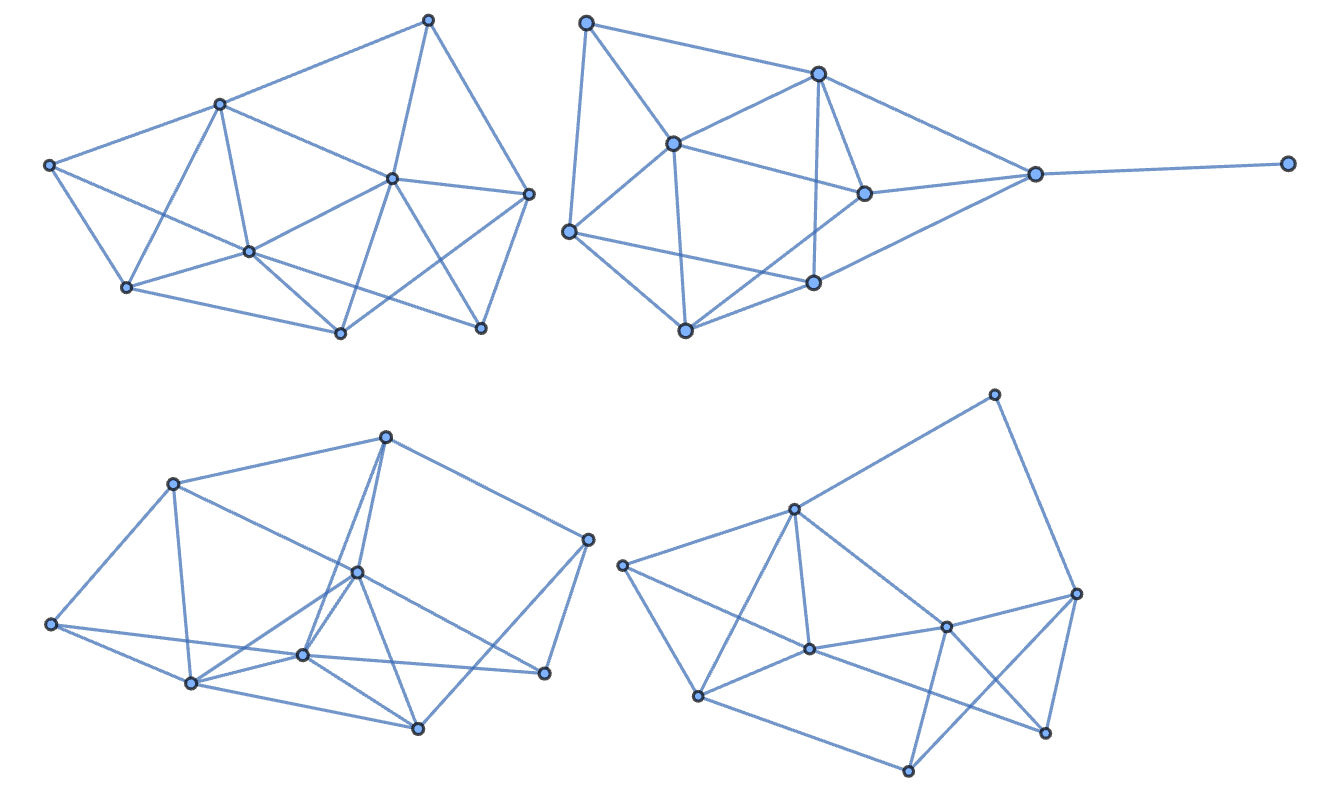}
\caption{$G_1,G_2,G_3,G_4$ in Example~\ref{ex1} (from left to right and top to bottom)}
\label{9}
\end{figure}
\begin{figure}[htbp]
\centering
\includegraphics[width=10cm]{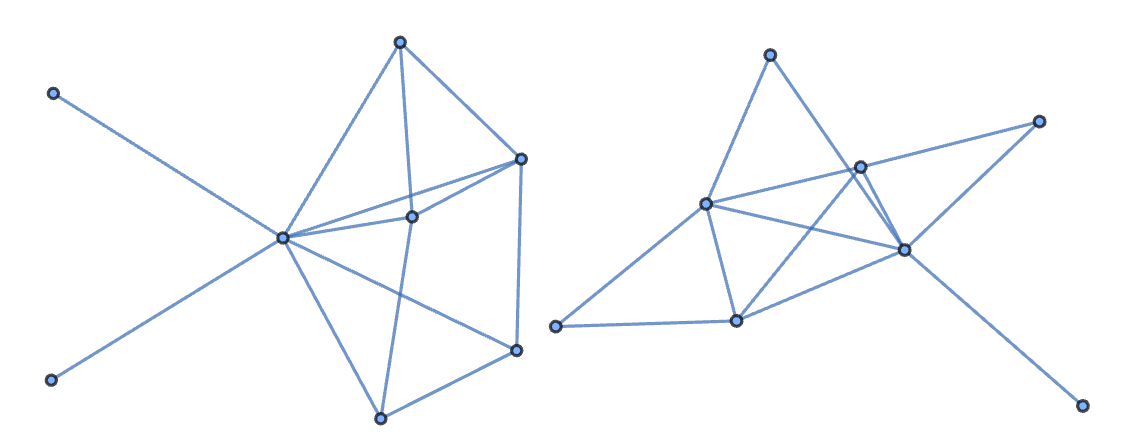}
\caption{$G_1,G_2$ in Example~\ref{ex2} (from left to right)}
\label{81}
\end{figure}
\begin{figure}[htbp]
\centering
\includegraphics[width=10cm]{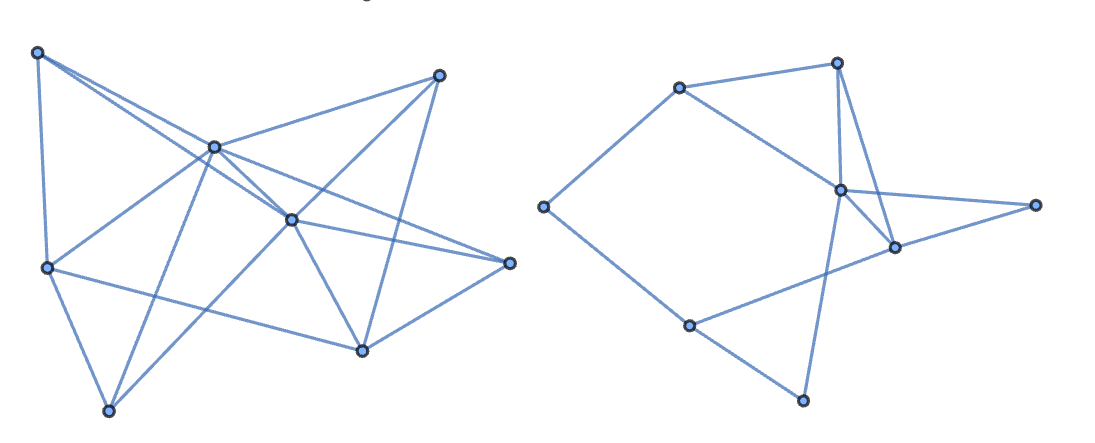}
\caption{$G_3,G_4$ in Example~\ref{ex2} (from left to right)}
\label{82}
\end{figure}

\section*{Acknowledgments}
The research of the third author is supported by National Key Research and Development Program of China 2023YFA1010203, National Natural Science Foundation of China (Grant No.\,12371357), and the fourth author is supported by Fundamental Research Funds for the Central Universities (Grant No. 531118010622),
National Natural Science Foundation of China (Grant No.\,1240011979) and Hunan
Provincial Natural Science Foundation of China (Grant No.\, 2024JJ6120).

\bibliographystyle{ieeetr}

\end{document}